\documentclass[12pt,twoside,reqno]{amsart}
\usepackage[colorlinks=true,citecolor=blue]{hyperref}
\usepackage{mathptmx, amsmath, amssymb, amsfonts, amsthm, mathptmx, enumerate, color}
\usepackage{graphicx}
\usepackage{epstopdf}
\usepackage{tikz}
\usepackage{pgfplots}

\newtheorem{theorem}{Theorem}[section]

\newtheorem{proposition}{Proposition}[section]
\theoremstyle{definition}
\newtheorem{definition}{Definition}[section]
\newtheorem{example}{Example}[section]
\newtheorem{remark}{Remark}[section]

\numberwithin{equation}{section}

\begin{document}
\setcounter{page}{1}

\vspace*{1.0cm}
\title[Optimality conditions and duality relations in nonsmooth FIMP]
{Optimality conditions and duality relations in nonsmooth fractional interval-valued multiobjective optimization}
\author[N.H. Hung, N.V. Tuyen]{ Nguyen Huy Hung$^{1}$ and Nguyen Van Tuyen$^{1,*}$}
\maketitle
\vspace*{-0.6cm}

\begin{center}
{\footnotesize {\it

$^1$Department of Mathematics, Hanoi Pedagogical University 2, Xuan Hoa, Phuc Yen, Vinh Phuc, Vietnam
}}\end{center}

\vskip 4mm {\small\noindent {\bf Abstract.}
This paper deals with Pareto solutions of a nonsmooth fractional interval-valued multiobjective optimization.  We first introduce four types of  Pareto solutions of the considered problem by considering the lower-upper interval order relation and then apply some advanced tools of variational analysis and generalized differentiation to establish necessary optimality conditions for these solutions. Sufficient conditions for  Pareto solutions of such a problem are also provided by means of introducing the concepts of (strictly) generalized convex functions defined in terms of the limiting/Mordukhovich subdifferential of locally Lipschitzian functions. Finally, a Mond--Weir type dual model is formulated, and weak, strong and converse-like duality relations are examined.

}

\renewcommand{\thefootnote}{}
\footnotetext{ $^*$Corresponding author.
\par
E-mail addresses: nguyenhuyhung@hpu2.edu.vn (N.H. Hung), nguyenvantuyen83@hpu2.edu.vn (N.V. Tuyen).
\par
Received ......; Accepted ....... }

\section{Introduction}

In this paper, we are interested in Pareto solutions of the following fractional multiobjective problem with   interval-valued objective functions:  
\begin{align}\label{problem}\tag{FIMP}
	LU-&\mathrm{Min}\;F(x):=\left(\frac{f_1(x)}{g_1(x)}, \ldots, \frac{f_m(x)}{g_m(x)}\right) 
	\\
	&\text{s.t.}\ \ x\in \Omega:=\{x\in S\,:\, h_j(x)\leq 0, j=1, \ldots, p\}, \notag
\end{align}
where $f_i, g_i\colon \mathbb{R}^n\to \mathcal{K}_c$, $i\in I:=\{1, \ldots, m\}$, are interval-valued functions defined respectively by $f_i(x)=[f_i^L(x), f_i^U(x)]$, $g_i(x)=[g_i^L(x), g_i^U(x)]$, $f_i^L,$ $f_i^U$, $g_i^L$,$ g_i^U\colon\mathbb{R}^n\to \mathbb{R}$  are locally Lipschitzian functions satisfying $f_i^L(x)\leq f_i^U(x)$ and 
$$0<g_i^L(x)\leq g_i^U(x)$$ for all $x\in S$ and $i\in I$, $\mathcal{K}_c$ is the class of all closed and bounded intervals in $\mathbb{R}$, i.e., 
$$\mathcal{K}_c=\{[a^L, a^U]\,:\, a^L, a^U\in\mathbb{R},\ \ a^L\leq a^U\},$$ 
$h_j\colon\mathbb{R}^n\to \mathbb{R}$, $j\in J:=\{1, \ldots, p\}$, are locally Lipschitzian functions, and $S$ is a nonempty and closed subset of $\mathbb{R}^n$.  

An interval-valued optimization problem is one of the deterministic optimization models to deal with the uncertain/incomplete data. More precisely, in  interval-valued optimization, the coefficients of objective and constraint functions are taken as closed intervals.
The study of optimality conditions and duality relations for  optimization problems with one or multiple interval-valued objective functions have recently received increasing interest in optimization community; see, e.g., \cite{Ahmad,Chalco-Cano et.al.-13,Hung-Tuan-Tuyen-22,Ishibuchi-Tanaka-90,Jennane,Kumar,Kummari-20,Luu-18,Osuna-Gomez-17,Qian,Singh-Dar-15,Singh-Dar-Kim-16,Singh-Dar-Kim-19,Su-2020,Tung-2019,Tung-2019b,Tuyen-2021,Wu-07,Wu-09,Wu-09-b,Wu-09-c,Zhang}  and the references therein. However, there are quite few publications devoted to  optimality conditions and duality relations for  fractional interval-valued optimization  problems; see \cite{Debnath-19,Dar-2021,Debnath-20}.

Debnath and Gupta \cite{Debnath-19} presented some necessary and sufficient optimality conditions for nondifferentiable fractional interval-valued  programming problems, where numerators of the objective function and constraint functions are convex, while denominators of the objective function are concave. Recently, these results have been extended to fractional interval-valued multiobjective problems \cite{Dar-2021,Debnath-20}. However, to the best of our knowledge, so far there have been no papers investigating optimality conditions and duality for fractional interval-valued multiobjective problems with locally Lipschitzian data.  

Motivated by the above observations, in this paper, we introduce some kinds of Pareto solutions with respect to lower-upper ($LU$) interval order relation for problems of the form \eqref{problem}. Then we employ the limiting/Mordukhovich subdifferential and the limiting/Mordukhovich normal cone  to derive necessary and sufficient optimality conditions for these Pareto solutions of problem \eqref{problem}. Along with optimality conditions, we state a dual problem in the sense of Mond--Weir to the primal one and examine weak, strong and converse duality relations under assumptions of (strictly) generalized convexity (cf. \cite{Bae-2022,Chuong-Kim-16-2,Chuong-16}). In addition, some examples are also given for analyzing the obtained results.

The paper is organized as follows. Section \ref{Preliminaries} contains some basic  definitions from variational analysis, interval analysis and several auxiliary results. In Section \ref{Optimiality-conditions}, we first introduce four kinds of Pareto solutions of problem \eqref{problem} and then establish necessary conditions for these solutions.  Sufficient optimality conditions for such solutions are provided by means of introducing (strictly) generalized convex functions defined in terms of the limiting subdifferential for locally Lipschitzian functions.  Section \ref{Duality-Relations} is devoted to describing duality relations.

\section{Preliminaries}\label{Preliminaries}

Throughout the paper, let $\mathbb{R}^n$ be the $n$-dimensional Euclidean space and  $\mathbb{R}^n_+$ be its nonnegative orthant.  The topological closure of a set $S$ is denoted  by  $\mathrm{cl}\,{S}$. As usual, the polar cone of a  set $S\subset \mathbb{R}^n$ is defined by 
\begin{equation*} 
	S^\circ:=\{x^*\in\mathbb{R}^n\;:\; \langle x^*, x\rangle\leq 0 \ \ \ \forall x\in S\}.
\end{equation*}

\begin{definition}[{see \cite{mor06,mor2018}}]{\rm  Given $\bar x\in  \mbox{cl}\,S$. The set
		\begin{equation*}
			N(\bar x; S):=\{z^*\in \mathbb{R}^n:\exists
			x^k\stackrel{S}\longrightarrow \bar x, \varepsilon_k\to 0^+, z^*_k\to z^*,
			z^*_k\in {\widehat N_{\varepsilon
					_k}}(x^k; S),\ \ \forall k \in\mathbb{N}\}
		\end{equation*}
		is called the {\em limiting/Mordukhovich normal cone}  of $S$ at $\bar x$, where
		\begin{equation*}
			\widehat N_\varepsilon  (x; S):= \bigg\{ {z^*  \in {\mathbb{R}^n} \;:\;\limsup_{u\overset{S} \rightarrow x}
				\frac{{\langle z^* , u - x\rangle }}{{\parallel u - x\parallel }} \leq \varepsilon } \bigg\}
		\end{equation*}
		is the set of  {\em $\varepsilon$-normals} of $S$  at $x$ and  $u\xrightarrow {{S}} x$ means that $u \rightarrow x$ and $u \in S$.
	}
\end{definition}

\medskip
Let  $\varphi \colon \mathbb{R}^n \to  \overline{\mathbb{R}}:=[-\infty, \infty]$ be an {\em extended-real-valued function}. The {\em  epigraph}  and  {\em domain} of $\varphi$ are denoted, respectively, by
\begin{align*}
	\mbox{epi }\varphi&:=\{(x, \alpha)\in\mathbb{R}^n\times\mathbb{R} \,:\,  \varphi(x)\leq \alpha\},
	\\
	\mbox{dom }\varphi &:= \{x\in \mathbb{R}^n \,:\, \ \ |\varphi(x)|<+\infty \}.
\end{align*}

\begin{definition}[{see \cite{mor06,mor2018}}]{\rm   
		Let $\bar x\in \mbox{dom }\varphi$. 
\begin{enumerate}[(i)]
	\item The set
	\begin{align*}
		\partial \varphi (\bar x):=\{x^*\in \mathbb{R}^n \,:\, (x^*, -1)\in N((\bar x, \varphi (\bar x)); \mbox{epi }\varphi )\},
	\end{align*}
	is called the {\it limiting/Mordukhovich subdifferential}  of $\varphi$ at $\bar x$. If $\bar x\notin \mbox{dom }\varphi$, then we put $\partial \varphi (\bar x)=\emptyset$.
	\item The set $\partial^+\varphi(\bar x):=-\partial(-\varphi)(\bar x)$ is called the {\em upper subdifferential}  of $\varphi$ at $\bar x$.
\end{enumerate}

	}
\end{definition}

We now summarize some properties of  the limiting  subdifferential that will be used in the sequel.
\begin{proposition}[{see \cite[Theorem 3.36]{mor06}}]\label{sum-rule} Let $\varphi_l\colon\mathbb{R}^n\to\overline{\mathbb{R}}$, $l=1, \ldots, p$, $p\geq 2$, be lower semicontinuous around $\bar x$ and let all but one of these functions be locally Lipschitzian around $\bar x$. Then we have the following inclusion
	\begin{equation*}
		\partial (\varphi_1+\ldots+\varphi_p) (\bar x)\subset \partial  \varphi_1 (\bar x) +\ldots+\partial \varphi_p (\bar x).
	\end{equation*}
\end{proposition}

Recall that, the function  $\varphi$ is called {\em lower semicontinuous $($l.s.c.$)$ at} a point  $\bar x\in\mathrm{dom}\varphi$ if
\begin{equation*}
	\varphi(\bar x)\leq \liminf_{x\to\bar x} \varphi(x).
\end{equation*}

We say that $\varphi$ is l.s.c. around $\bar x$ when it is l.s.c. at any point of some neighborhood of $\bar{x}$. The function $\varphi$ is called {\em locally Lipschitzian around} $\bar{x}$, or {\em  Lipschitz continuous around} $\bar{x}$ if there is a
neighborhood $U$ of this point  and a constant $l\geq 0$  such that 
\begin{equation*}
\|\varphi(v)-\varphi(u)\|\leq l\|v-u\| \ \ \text{for all}\ \ u,v\in U.
\end{equation*}

\begin{proposition}[{see \cite[Theorem 3.46]{mor06}}]\label{max-rule}
	Let $\varphi_l\colon\mathbb{R}^n\to\overline{\mathbb{R}}$, $l=1, \ldots, p$,  be  locally Lipschitzian around $\bar x$. Then the function
	$ \phi(\cdot):=\max\{\varphi_l(\cdot):l=1, \ldots, p\}$
	is also locally Lipschitzian around $\bar x$ and one has
	\begin{equation*}
		\partial \phi(\bar x)\subset \bigcup\bigg\{\partial\bigg(\sum_{l=1}^{p}\lambda_l\varphi_l\bigg)(\bar x)\;:\; (\lambda_1, \ldots, \lambda_p)\in\Lambda(\bar x)\bigg\},
	\end{equation*}
	where 
	$$\Lambda(\bar x):=\big\{(\lambda_1, \ldots, \lambda_p)\;:\; \lambda_l\geq 0, \sum_{l=1}^{p}\lambda_l=1, \lambda_l[\varphi_l(\bar x)-\phi(\bar x)]=0\big\}.$$
\end{proposition}

\begin{proposition}[{\rm see \cite[ Corollary 1.111(ii)]{mor06}}]\label{quotion-rule}
	Let $\varphi_i\colon\mathbb{R}^n\to \overline{\mathbb{R}},$ $i=1, 2$, be locally Lipschitzian around $\bar x$. If $\varphi_2(\bar x)\neq 0$, then we have
	\begin{equation*}
		\partial \left(\frac{\varphi_1}{\varphi_2}\right)(\bar x)\subset \frac{\partial(\varphi_2(\bar{x})\varphi_1)(\bar{x})+\partial(-\varphi_1(\bar{x})\varphi_2)(\bar{x})}{[\varphi_2(\bar{x})]^2}.
	\end{equation*}  
\end{proposition}

\begin{proposition}[{\rm see \cite[Proposition 1.114]{mor06}}] \label{Fermat-rule} Let   $\varphi\colon\mathbb{R}^n\to\overline{\mathbb{R}}$  be finite at $\bar x$. If $\bar x$ is a local minimizer of  $\varphi$, then $ 0\in\partial\varphi(\bar x).$
\end{proposition}

Next we recall some definitions and facts in interval analysis. Let $A=[a^L, a^U]$ and $B=[b^L, b^U]$ be two intervals in $\mathcal{K}_c$. Then, we define
\begin{enumerate}[(i)]
	\item $A+B:=\{a+b\,:\, a\in A, b\in B\}=[a^L+b^L, a^U+b^U]$;
	\item $A-B:=\{a-b\,:\, a\in A, b\in B\}=[a^L-b^U, a^U-b^L]$;
	\item For each $k\in\mathbb{R}$,
	\begin{equation*}kA:=\{ka\,:\, a\in A\}=
		\begin{cases}
			[ka^L, ka^U]\ \ \text{if}\ \ k\geq 0,
			\\
			[ka^U, ka^L]\ \ \text{if}\ \ k < 0; 
		\end{cases}
	\end{equation*}
\item $\frac{A}{B}:=\left[\min\left(\frac{a^L}{b^L}, \frac{a^L}{b^U},\frac{a^U}{b^L}, \frac{a^U}{b^U}\right), \max\left(\frac{a^L}{b^L}, \frac{a^L}{b^U},\frac{a^U}{b^L}, \frac{a^U}{b^U}\right) \right]$, if $0\notin B$,
\end{enumerate}
see, e.g., \cite{Alefeld,Moore-1966,Moore-1979}  for more details.

\begin{definition}[{see \cite{Ishibuchi-Tanaka-90,Wu-07}}]{\rm 
		Let $A=[a^L, a^U]$ and $B=[b^L, b^U]$ be two intervals in $\mathcal{K}_c$. We say that:
		\begin{enumerate}[(i)]
			\item $A\leq_{LU} B$ if $a^L\leq b^L$ and $a^U\leq b^U$.
			
			\item  $A<_{LU} B$ if $A\leq_{LU} B$ and $A\neq B$, or, equivalently, $A<_{LU} B$ if
			\\
			$ 
			\begin{cases}
				a^L<b^L
				\\
				a^U\leq b^U,
			\end{cases}
			$ 
			or \ \ \ \
			$ 
			\begin{cases}
				a^L\leq b^L
				\\
				a^U< b^U,
			\end{cases}
			$ 
			or \ \ \ \
			$ 
			\begin{cases}
				a^L< b^L
				\\
				a^U< b^U.
			\end{cases}
			$ 
			\item $A<^s_{LU} B$ if $a^L<b^L$ and $a^U<b^U$.
		\end{enumerate}
	 }
\end{definition}
\section{Optimality conditions}\label{Optimiality-conditions}

We now introduce Pareto solutions of \eqref{problem} with respect to $LU$ interval order relation. For the sake of convenience, {\em we always assume hereafter} that $f_i^L(x)\geq 0$, $\forall x\in S$ and $i\in I$.  For each $i\in I$ and $x\in\mathbb{R}^n$, put $F_i(x):=\frac{f_i(x)}{g_i(x)}$. By definition, we have
\begin{equation*}
	F_i(x):=\frac{f_i(x)}{g_i(x)}=\left[\frac{f_i^L(x)}{g_i^U(x)}, \frac{f_i^U(x)}{g_i^L(x)}\right].
\end{equation*}

\begin{definition}\label{Defi-solution}{\rm
		Let $\bar x\in \Omega$. We say that:
		\begin{enumerate}[(i)]
			\item $\bar x$ is a {\em type-1 Pareto solution} of \eqref{problem}, denoted by $\bar x\in \mathcal{S}_1\eqref{problem}$, if there is no $x\in \Omega$ such that
			\begin{equation*}
				\begin{cases}
					F_i(x)\leq_{LU} F_i(\bar x),\ \ &\forall i\in I,
					\\
					F_k(x)<_{LU} F_k(\bar x),\ \ &\text{for at least one}\ \ k\in I.
				\end{cases} 
			\end{equation*}
			
			\item $\bar x$ is a {\em type-2 Pareto solution} of \eqref{problem}, denoted by $\bar x\in \mathcal{S}_2\eqref{problem}$, if there is no $x\in \Omega$ such that
			\begin{equation*}
				\begin{cases}
					F_i(x)\leq_{LU} F_i(\bar x),\ \ &\forall i\in I,
					\\
					F_k(x)<^s_{LU}F_k(\bar x),\ \ &\text{for at least one}\ \ k\in I.
				\end{cases} 
			\end{equation*}
			
			\item  $\bar x$ is a {\em type-1 weakly Pareto solution} of \eqref{problem}, denoted by $\bar x\in\mathcal{S}_1^{w}\eqref{problem}$, if there is no $x\in \Omega$ such that
			\begin{equation*}
				F_i(x)<_{LU} F_i(\bar x),\ \ \forall i\in I.
			\end{equation*}
			
			\item $\bar x$ is a {\em type-2 weakly Pareto solution} of \eqref{problem}, denoted by $\bar x\in\mathcal{S}_2^{w}\eqref{problem}$, if there is no $x\in \Omega$ such that
			\begin{equation*}
				F_i(x)<^s_{LU} F_i(\bar x),\ \ \forall i\in I.
			\end{equation*}    
		\end{enumerate}
	}
\end{definition}

\begin{remark}\label{Remark-3.1}
	{\rm The following relations are immediate from the definition of Pareto solutions.
		\begin{enumerate}[(i)]
			\item $\mathcal{S}_1\eqref{problem} \subseteq\mathcal{S}_2\eqref{problem} \subseteq \mathcal{S}_2^{w}\eqref{problem}$. 
			
			\item 	$\mathcal{S}_1\eqref{problem} \subseteq \mathcal{S}_1^{w}\eqref{problem} \subseteq \mathcal{S}_2^{w}\eqref{problem}$. 			
		\end{enumerate}	 
Furthermore, 	the above inclusions may be strict; see, e.g., \cite[Examples 3.3--3.5 ]{Tung-2019}.}
\end{remark}

The following result provides a Fritz-John type necessary condition for  type-2 weakly Pareto solutions of problem \eqref{problem}.

\begin{theorem}\label{Necessary-Theorem} If $\bar x\in \mathcal{S}_2^{w}\eqref{problem}$, then there exist $\lambda_i^L\geq 0$,  $\lambda_i^U\geq 0$, $i\in I$, and $\mu_j\geq 0$, $j\in J$ with $\sum_{i\in I}(\lambda_i^L+\lambda_i^U)+\sum_{j\in J}\mu_j=1$, such that
	\begin{align} 
		&0\in \sum_{i\in I}\frac{\lambda_i^L}{g_i^U(\bar x)} \left[\partial f_i^L(\bar x)-\frac{f_i^L(\bar x)}{g_i^U(\bar x)}\partial^+ g_i^U(\bar x)\right]+\sum_{i\in I}\frac{\lambda_i^U}{g_i^L(\bar x)} \left[\partial f_i^U(\bar x)-\frac{f_i^U(\bar x)}{g_i^L(\bar x)}\partial^+ g_i^L(\bar x)\right]\notag
		\\
		&\ \ \ \ \ \ \ \ \ \ \ \ \ \ \ \ \ \ \ \ \ \ \ \ \ \ \ \ \ \ \ \ \ \ \ \ \ \ \ \ \ \ +\sum_{j\in J}\mu_j\partial h_j(\bar x)+N(\bar x; S),\ \ \mu_j h_j(\bar x)=0, \ \ j\in J. \label{Necessary-conditions}
	\end{align}
\end{theorem}
\begin{proof} Since $\bar x\in \mathcal{S}_2^{w}\eqref{problem}$, there is no $x\in\Omega$ such that $F_i(x)<^s_{LU} F_i(\bar x)$, $\forall i\in I$, i.e., 
	\begin{equation*}
		\frac{f_i^L(x)}{g_i^U(x)}< \frac{f_i^L(\bar x)}{g_i^U(\bar x)}\ \ \text{and}\ \ \frac{f_i^U(x)}{g_i^L(x)}< \frac{f_i^U(\bar x)}{g_i^L(\bar x)}, \ \ \forall i\in I.  
	\end{equation*}	 
	Hence for each $x\in\Omega$, there exists $i\in I$ such that	
	\begin{equation}\label{equ-1}
		\frac{f_i^L(x)}{g_i^U(x)}\geq \frac{f_i^L(\bar x)}{g_i^U(\bar x)}\ \ \text{or}\ \ \frac{f_i^U(x)}{g_i^L(x)}\geq \frac{f_i^U(\bar x)}{g_i^L(\bar x)}.   
	\end{equation}	 
	Let $\varphi$  be a real-valued function defined by
	$$\varphi(x):=\max_{i\in I, j\in J}\left\{\frac{f_i^L(x)}{g_i^U(x)}-\frac{f_i^L(\bar x)}{g_i^U(\bar x)}, \frac{f_i^U(x)}{g_i^L(x)}- \frac{f_i^U(\bar x)}{g_i^L(\bar x)}, h_j(x)\right\},\ \ \forall x\in\mathbb{R}^n.$$
	By \eqref{equ-1}, we have 
	$$0=\varphi(\bar x)\leq \varphi(x), \ \ \forall x\in \Omega.$$
	If $x\in S\setminus\Omega$, then there exists $j_0\in J$ such that $h_{j_0}(x)>0$ and so $\varphi(x)>0$.  This implies that
	$$0=\varphi(\bar x)\leq \varphi(x), \ \ \forall x\in S,$$ 
	or, equivalently, $\bar x$ is a minimizer to the following unconstrained optimization problem 
	\begin{equation*}
		\text{minimizer}\ \ \varphi(x)+\delta(x; S),\ \ x\in\mathbb{R}^n,
	\end{equation*} 
	where $\delta(\cdot,; S)$ is the indicator function of $\Omega$ and defined by 
	\begin{equation*}
		\delta(x; S)=
		\begin{cases}
			0, &\ \ \text{if}\ \ x\in S,
			\\
			+\infty, &\ \ \text{otherwise}.
		\end{cases}
	\end{equation*}
	By Proposition \ref{Fermat-rule}, we have
	\begin{equation*}
		0\in \partial (\varphi+\delta(\cdot\,; S))(\bar x).
	\end{equation*}
	Clearly, $\varphi$ is   locally Lipschitzian around $\bar x$ and $\delta(\cdot\,; S)$ is lower semicontinuous around this point. Hence by Proposition \ref{sum-rule} and the fact that $\partial \delta(\cdot\,; S)(\bar x)=N(\bar x; S)$ (see  e.g., \cite[Proposition 1.19]{mor06}), we obtain 
	\begin{equation}\label{equ:2}
		0\in \partial \varphi(\bar x)+N(\bar x; S).
	\end{equation}   
	By Proposition \ref{max-rule}, we have
	\begin{align}
		\partial \varphi(\bar x) &\subset \Bigg\{\sum_{i\in I}\lambda_i^L \partial\left(\frac{f_i^L}{g_i^U}\right) (\bar x)+\sum_{i\in I}\lambda_i^U \partial \left(\frac{f_i^U}{g_i^L}\right) (\bar x)+\sum_{j\in J}\mu_j\partial h_j(\bar x)\;:\; \lambda_i^L, \lambda_i^U\geq 0, i\in I,\notag 
		\\
		&\ \ \ \ \ \ \ \ \ \ \ \ \ \ \ \ \ \ \ \ \ \mu_j\geq 0, j\in J, \sum_{i\in I}(\lambda_i^L+\lambda_i^U)+\sum_{j\in J}\mu_j=1, \mu_jh_j(\bar x)=0, j\in J\Bigg\}.\label{equa-4}
	\end{align}
	Now, taking Proposition \ref{quotion-rule} into account, we arrive at
	\begin{align}
		\partial\left(\frac{f_i^L}{g_i^U}\right) (\bar x)&\subset \frac{\partial (g_i^U(\bar x)f_i^L)(\bar x)+\partial(-f_i^L(\bar x)g_i^U)(\bar x)}{[g_i^U(\bar x)]^2}\notag
		\\
		&=\frac{g_i^U(\bar x)\partial f_i^L(\bar x)+f_i^L(\bar x)\partial(-g_i^U)(\bar x)}{[g_i^U(\bar x)]^2},  \notag
		\\
		&=\frac{g_i^U(\bar x)\partial f_i^L(\bar x)-f_i^L(\bar x)\partial^+g_i^U(\bar x)}{[g_i^U(\bar x)]^2}, \ \ \forall i\in I,\label{equa-5}
	\end{align}
	where the last equalities hold due to the fact that $f_i^L(\bar x)\geq 0$, $g_i^U(\bar x)>0$, and 
	$$\partial(-g_i^U)(\bar x)=-\partial^+g_i^U(\bar x), \ \ \forall i\in I.$$
Similarly, we have
	\begin{equation}\label{equa-6}
		\partial\left(\frac{f_i^U}{g_i^L}\right) (\bar x)\subset \frac{g_i^L(\bar x)\partial f_i^U(\bar x)-f_i^U(\bar x)\partial^+ g_i^L(\bar x)}{[g_i^L(\bar x)]^2}, \ \ \forall i\in I.
	\end{equation}
	It follows from \eqref{equ:2}--\eqref{equa-6} that
	\begin{align*}
		0\in \Bigg\{\sum_{i\in I}\frac{\lambda_i^L}{g_i^U(\bar x)} &\left[\partial f_i^L(\bar x)-\frac{f_i^L(\bar x)}{g_i^U(\bar x)}\partial^+ g_i^U(\bar x)\right]+\sum_{i\in I}\frac{\lambda_i^U}{g_i^L(\bar x)} \left[\partial f_i^U(\bar x)-\frac{f_i^U(\bar x)}{g_i^L(\bar x)}\partial^+ g_i^L(\bar x)\right]
		\\
		&+\sum_{j\in J}\mu_j\partial h_j(\bar x)\,:\lambda_i^L, \lambda_i^U\geq 0, i\in I, \mu_j\geq 0, j\in J, \notag 
		\\
		&\ \ \ \ \sum_{i\in I}(\lambda_i^L+\lambda_i^U)+\sum_{j\in J}\mu_j=1, \mu_jh_j(\bar x)=0, j\in J\Bigg\}+N(\bar x; S).
	\end{align*}
	In other words, there exist $\lambda_i^L\geq 0,$ $\lambda_i^U\geq 0$, $i\in I$, and $\mu_j\geq 0$, $j\in J$, with $\sum_{i\in I}(\lambda_i^L+\lambda_i^U)+\sum_{j\in J}\mu_j=1$ satisfying \eqref{Necessary-conditions}. The proof is complete. 
\end{proof}

The relation obtained in \eqref{Necessary-conditions} suggests us to define a  Karush--Kuhn--Tucker (KKT) type condition when dealing with Pareto solutions of problem \eqref{problem}.

\begin{definition}
	{\rm Let $\bar x\in\Omega$. We say that $\bar x$ satisfies the {\em  KKT condition} if \eqref{Necessary-conditions} holds with $\lambda_i^L\geq 0$,  $\lambda_i^U\geq 0$, $i\in I$, and $\mu_j\geq 0$, $j\in J$ such that $\sum_{i\in I}(\lambda_i^L+\lambda_i^U)+\sum_{j\in J}\mu_j=1$ and  $(\lambda^L,\lambda^U)\neq (0,0)$, where $\lambda^L:=(\lambda_1^L, \ldots, \lambda_m^L)$ and  $\lambda^U:=(\lambda_1^U, \ldots, \lambda_m^U)$. 
	}
\end{definition}

In order to obtain  optimality conditions of KKT-type for Pareto solutions of problem \eqref{problem}, we use the following well known constraint qualification.

\begin{definition}[see \cite{Chuong-Kim-16}]{\rm 
		Let $\bar x\in \Omega$. We say that the {\em constraint qualification}  \eqref{CQ} is satisfied at  $\bar x$ if there do not exist $\mu_j\geq 0$, $j\in J(\bar x)$ not all zero, such that
		\begin{equation}\label{CQ}\tag{CQ}
			0\in\sum_{j\in J(\bar x)}\mu_j\partial h_j(\bar x)+N(\bar x; S),
		\end{equation}
		where $J(\bar x):=\{j\in J\;:\; g_j(\bar x)=0\}$.}
\end{definition}

It is worth to mentioning here that the above \eqref{CQ} reduces to the classical Mangasarian--Fromovitz constraint qualification when the functions $h_1, \ldots, h_p$ are strictly differentiable at such $\bar x$ and $S=\mathbb{R}^n$; see e.g., \cite{mor06-2,Chuong-09}.
\begin{theorem}\label{KKT-Theorem}
	If $\bar x\in \mathcal{S}_2^{w}\eqref{problem}$ and the \eqref{CQ}  holds at $\bar x$, then $\bar x$ satisfies the KKT condition.    
\end{theorem}   
\begin{proof} Assume that $\bar x\in \mathcal{S}_2^{w}\eqref{problem}$ and the \eqref{CQ} holds at $\bar x$.  Then by Theorem \ref{Necessary-Theorem}, there exist $\lambda_i^L\geq 0$,  $\lambda_i^U\geq 0$, $i\in I$, and $\mu_j\geq 0$, $j\in J$ with $\sum_{i\in I}(\lambda_i^L+\lambda_i^U)+\sum_{j\in J}\mu_j=1$ satisfying \eqref{Necessary-conditions}. If $(\lambda^L,\lambda^U)= (0,0)$, then $0\in \sum_{j\in J}\mu_j\partial h_j(\bar x)+N(\bar x; S)$ and $\mu_j h_j(\bar x)=0$ for all $j\in J$. Hence, by the condition \eqref{CQ}, $\mu_j=0$ for all $j\in J$. This contradicts to the fact that 
$$\sum_{i\in I}(\lambda_i^L+\lambda_i^U)+\sum_{j\in J}\mu_j=1.$$
Therefore, $(\lambda^L,\lambda^U)\neq (0,0)$.  The proof is complete.
\end{proof}

{The following  example shows that the conclusion of Theorem \ref{KKT-Theorem} may fail if the \eqref{CQ} is	not satisfied.}
\begin{example}\label{Not-CQ} We consider problem \eqref{problem} with $m=2,$ $n=p=1$, $S=(-\infty, 0]$, $h(x)=x^2$, $f^L_1(x)=f_1^U(x)=-x+2$, $f^L_2(x)=f_2^U(x)=-x+3$, $g_1^L(x)=g_2^L(x)=-2x+1$, and $g_1^U(x)=g_2^U(x)=-2x+2$. Then
\begin{equation*}
F_1(x)=\left[\frac{-x+2}{-2x+ 2}, \frac{-x+2}{-2x+1}\right],  F_2(x)=\left[\frac{-x+3}{-2x+ 2}, \frac{-x+3}{-2x+1}\right],
\end{equation*} 
and $\Omega=\{0\}$. Clearly, $\bar x=0\in \mathcal{S}_2^w\eqref{problem}$, $\nabla h(\bar x)=0$, and $N(\bar x, S)=[0, +\infty)$. By Theorem \ref{Necessary-Theorem}, there exist $\lambda_i^L\geq 0$, $\lambda_i^U\geq 0$, $i\in I=\{1, 2\}$, $\mu\geq 0$ with $\sum_{i\in I}(\lambda_i^L+\lambda_i^U)+\mu=1$ satisfying \eqref{Necessary-conditions}, i.e.,
\begin{equation*}
0\in \frac{\lambda_1^L}{2}+\lambda_2^L+3\lambda_1^U+5\lambda_2^U +[0, +\infty).
\end{equation*}
This implies that $\lambda_1^L=\lambda_2^L=\lambda_1^U=\lambda_2^U=0$ and so the KKT condition do not hold at $\bar x$. Actually, the \eqref{CQ} fails to hold at $\bar x$. 
\end{example}

Next we present sufficient conditions for Pareto solutions of \eqref{problem}.  In order to obtain these sufficient conditions, we need to introduce concepts of (strictly) generalized convexity  at a given point for a family of locally Lipschitzian functions. The following definition is motivated from  \cite{Chuong-16}.

\begin{definition}
	{\rm 
		\begin{enumerate}[(i)]
			\item  We say that $(F, h)$ is {\em generalized  convex on $S$} at $\bar x\in S$ if for any $x\in S$, $x_i^{*L}\in\partial f_i^L(\bar x)$,  $x_i^{*U}\in\partial f_i^U(\bar x)$, $y_i^{*L}\in\partial^+ g_i^L(\bar x)$,  $y_i^{*U}\in\partial^+ g_i^U(\bar x)$,  $i\in I$, and $z^*_j\in\partial h_j(\bar x)$, $j\in J$, there exists $\nu\in [N(\bar x; S)]^\circ$ satisfying
			\begin{equation*}  
				\begin{split}
					&f_i^L(x) -f_i^L(\bar x)\geq \langle x_i^{*L}, \nu\rangle, \ \ \forall  i\in I,
					\\
					&f_i^U(x) -f_i^U(\bar x)\geq \langle x_i^{*U}, \nu\rangle, \ \ \forall  i\in I,
					\\
					&g_i^L(x) -g_i^L(\bar x)\leq \langle y_i^{*L}, \nu\rangle, \ \ \forall  i\in I,
					\\
					&g_i^U(x) -g_i^U(\bar x)\leq \langle y_i^{*U}, \nu\rangle, \ \ \forall  i\in I,
					\\
					&h_j(x)- h_j(\bar x) \geq \langle z^*_j, \nu\rangle,\ \ \forall j\in J.
				\end{split}
			\end{equation*}
			\item  	  
			We say that $(F, h)$ is {\em strictly generalized  convex on $S$} at $\bar x\in S$ if for any $x\in S\setminus\{\bar x\}$, $x_i^{*L}\in\partial f_i^L(\bar x)$,  $x_i^{*U}\in\partial f_i^U(\bar x)$, $y_i^{*L}\in\partial^+ g_i^L(\bar x)$,  $y_i^{*U}\in\partial^+ g_i^U(\bar x)$,  $i\in I$, and $z^*_j\in\partial h_j(\bar x)$, $j\in J$, there exists $\nu\in [N(\bar x; S)]^\circ$ satisfying
			\begin{equation*} 
				\begin{split}
					&f_i^L(x) -f_i^L(\bar x)> \langle x_i^{*L}, \nu\rangle, \ \ \forall  i\in I,
					\\
					&f_i^U(x) -f_i^U(\bar x)> \langle x_i^{*U}, \nu\rangle, \ \ \forall  i\in I,
					\\
					&g_i^L(x) -g_i^L(\bar x)\leq \langle y_i^{*L}, \nu\rangle, \ \ \forall  i\in I,
					\\
					&g_i^U(x) -g_i^U(\bar x)\leq \langle y_i^{*U}, \nu\rangle, \ \ \forall  i\in I,
					\\
					&h_j(x)- h_j(\bar x) \geq \langle z^*_j, \nu\rangle,\ \ \forall j\in J.
				\end{split}
			\end{equation*}
		\end{enumerate}
	}
\end{definition}
\begin{remark}
	{\rm 
		We see that if $S$ is convex and $f^L_i$, $f^U_i$, $-g^L_i$, $-g^U_i$, $i\in I$, and $h_j$, $j\in J$, are convex, then $(F, h)$ is generalized convex  on $S$ at any $\bar x\in S$ with $\nu=x-\bar x$. Moreover, the class of generalized convex functions is properly larger than the one of convex functions; see, e.g., \cite[Example 3.2]{Chuong-14} and \cite[Example 3.12]{Chuong-Kim-16}.  
	}
\end{remark}

\begin{theorem}\label{Sufficient-Theorem} Let $\bar x\in\Omega$ satisfy the KKT condition. 
	\begin{enumerate}[\rm(i)]
		\item If $(F, h)$ is generalized convex on $S$ at $\bar x$, then $\bar x\in\mathcal{S}_2^{w}\eqref{problem}$.
		
		\item If $(F, h)$ is strictly  generalized convex on $S$ at $\bar x$, then $\bar x\in \mathcal{S}_1\eqref{problem}$ and so $\bar x\in\mathcal{S}_2\eqref{problem}$ and $\bar x\in\mathcal{S}_1^{w}\eqref{problem}$.  
	\end{enumerate}
\end{theorem}	
\begin{proof} Since $\bar x$ satisfies the KKT condition, there exist $(\lambda^L,\lambda^U)\in (\mathbb{R}^m_+\times\mathbb{R}^m_+)\setminus\{(0,0)\}$, $\mu_j\geq 0$, $j\in J$, and  $x_i^{*L}\in \partial f_i^L(\bar x)$, $x_i^{*U}\in \partial f_i^U(\bar x)$, $y_i^{*L}\in\partial^+ g_i^L(\bar x)$, $y_i^{*U}\in\partial^+ g_i^U(\bar x)$, $i\in I$, $z_j^*\in \partial h_j(\bar x)$, $j\in J$, and $\omega^*\in N(\bar x; S)$ such that $\mu_jh_j(\bar x)=0, j\in J$ and
	\begin{equation*}
		\sum_{i\in I}\frac{\lambda_i^L}{g_i^U(\bar x)} \left[x_i^{*L}-\frac{f_i^L(\bar x)}{g_i^U(\bar x)}y_i^{*U}\right]+\sum_{i\in I}\frac{\lambda_i^U}{g_i^L(\bar x)} \left[x_i^{*U}-\frac{f_i^U(\bar x)}{g_i^L(\bar x)}y_i^{*L}\right]+\sum_{j\in J}\mu_j z_j^*+\omega^*=0,
	\end{equation*}	
	or, equivalently,
	\begin{equation}\label{equ:5}
		\sum_{i\in I}\frac{\lambda_i^L}{g_i^U(\bar x)} \left[x_i^{*L}-\frac{f_i^L(\bar x)}{g_i^U(\bar x)}y_i^{*U}\right]+\sum_{i\in I}\frac{\lambda_i^U}{g_i^L(\bar x)} \left[x_i^{*U}-\frac{f_i^U(\bar x)}{g_i^L(\bar x)}y_i^{*L}\right]+\sum_{j\in J}\mu_j z_j^*=-\omega^*.
	\end{equation}
	First, we prove (i). Assume on the contrary that $\bar x\notin \mathcal{S}_2^w$. This means that there exists $\hat x\in\Omega$ such that $F_i(\hat x)<^s_{LU} F_i(\bar x)$, $\forall i\in I$, or, equivalently, 
	\begin{equation}\label{equ:4}
		\frac{f_i^L(\hat x)}{g_i^U(\hat x)}< \frac{f_i^L(\bar x)}{g_i^U(\bar x)}\ \ \text{and}\ \ \frac{f_i^U(\hat x)}{g_i^L(\hat x)}< \frac{f_i^U(\bar x)}{g_i^L(\bar x)}, \ \ \forall i\in I.  
	\end{equation} 
	By the generalized convexity of $(F, h)$ at $\bar x$, for such $\hat x$, there is $\nu\in [N(\bar x; S)]^\circ$  such that
	\begin{equation*} 
		\begin{aligned}
			&\sum_{i\in I}\frac{\lambda_i^L}{g_i^U(\bar x)} \left[\langle x_i^{*L},\nu\rangle-\frac{f_i^L(\bar x)}{g_i^U(\bar x)}\langle y_i^{*U},\nu\rangle \right]+\sum_{i\in I}\frac{\lambda_i^U}{g_i^L(\bar x)} \left[\langle x_i^{*U},\nu\rangle-\frac{f_i^U(\bar x)}{g_i^L(\bar x)}\langle y_i^{*L},\nu\rangle\right]+\sum_{j\in J}\mu_j \langle z_j^*,\nu\rangle
			\\
			&\leq \sum_{i\in I}\frac{\lambda_i^L}{g_i^U(\bar x)} \left[f_i^L(\hat x)-f_i^L(\bar x)-\frac{f_i^L(\bar x)}{g_i^U(\bar x)}(g_i^U(\hat x)-g_i^U(\bar x)) \right]
			\\
			&\ \ \ \ +\sum_{i\in I}\frac{\lambda_i^U}{g_i^L(\bar x)} \left[f_i^U(\hat x)-f_i^U(\bar x)-\frac{f_i^U(\bar x)}{g_i^L(\bar x)}(g_i^L(\hat x)-g_i^L(\bar x))\right]+\sum_{j\in J}\mu_j (h_j(\hat x)-h_j(\bar x))
			\\
			&=\sum_{i\in I}\frac{\lambda_i^L}{g_i^U(\bar x)} \left[f_i^L(\hat x)-\frac{f_i^L(\bar x)}{g_i^U(\bar x)}g_i^U(\hat x) \right]+\sum_{i\in I}\frac{\lambda_i^U}{g_i^L(\bar x)} \left[f_i^U(\hat x)-\frac{f_i^U(\bar x)}{g_i^L(\bar x)}g_i^L(\hat x)\right]
			\\
			&\ \ \ \ \ \ \ \ \ \ \ \ \ \ \ \ \ \ \ \ \ \ \ \ \ \ \ \ \ \ \ \ \ \ \ \ \ \ \ \ \ \ \ \ \ \ \  \ \ \ \ \ \ \ \ \ \ +\sum_{j\in J}\mu_j (h_j(\hat x)-h_j(\bar x)).
		\end{aligned}
	\end{equation*}   
	It follows from \eqref{equ:5} and relations $\omega^*\in N(\bar x; S)$ and $\nu\in [N(\bar x, S)]^\circ$ that
	\begin{align}
		0&\leq \langle -\omega^*,\nu\rangle=\sum_{i\in I}\frac{\lambda_i^L}{g_i^U(\bar x)} \left[\langle x_i^{*L},\nu\rangle-\frac{f_i^L(\bar x)}{g_i^U(\bar x)}\langle y_i^{*U},\nu\rangle \right]+\sum_{i\in I}\frac{\lambda_i^U}{g_i^L(\bar x)} \left[\langle x_i^{*U},\nu\rangle-\frac{f_i^U(\bar x)}{g_i^L(\bar x)}\langle y_i^{*L},\nu\rangle\right]\notag
		\\
		&\ \ \ \ \ \ \ \ \ \ \ \ \ \ \ \ \ \ \ \ \ \ \ \ \ \ \ \ \ \ \ \ \ \ \ \ \ \ \ \ \ \ \ \ \ \ \ \ \ \ \ \ \ \ \ \ \ \ \ \ \ \ \ \ \ \ \  +\sum_{j\in J}\mu_j \langle z_j^*,\nu\rangle\notag
		\\
		&\ \ \ \ \ \ \ \ \ \ \ \ \ \ \ \ \leq \sum_{i\in I}\frac{\lambda_i^L}{g_i^U(\bar x)} \left[f_i^L(\hat x)-\frac{f_i^L(\bar x)}{g_i^U(\bar x)}g_i^U(\hat x) \right]+\sum_{i\in I}\frac{\lambda_i^U}{g_i^L(\bar x)} \left[f_i^U(\hat x)-\frac{f_i^U(\bar x)}{g_i^L(\bar x)}g_i^L(\hat x)\right]\notag
		\\
		&\ \ \ \ \ \ \ \ \ \ \ \ \ \ \ \ \ \ \ \ \ \ \ \ \ \ \ \ \ \ \ \ \ \ \ \ \ \ \ \ \ \ \ \ \ \ \ \ \ \ \ \ \ \ \ \ \ \ \ \ \ \ \ \ \ \ \ \ \ \ \ \ +\sum_{j\in J}\mu_j (h_j(\hat x)-h_j(\bar x)).\label{equa:5}
	\end{align} 
	Furthermore, we see that $\mu_j h_j(\bar x)=0$ and $\mu_j h_j(\hat x)\leq 0$ for all $j\in J$. This and \eqref{equa:5} imply that
	\begin{equation*}
		\sum_{i\in I}\frac{\lambda_i^L}{g_i^U(\bar x)} \left[f_i^L(\hat x)-\frac{f_i^L(\bar x)}{g_i^U(\bar x)}g_i^U(\hat x) \right]+\sum_{i\in I}\frac{\lambda_i^U}{g_i^L(\bar x)} \left[f_i^U(\hat x)-\frac{f_i^U(\bar x)}{g_i^L(\bar x)}g_i^L(\hat x)\right]\geq 0.
	\end{equation*}
	Since $(\lambda^L,\lambda^U)\neq (0,0)$, there exists $i_0\in I$ such that
	\begin{equation*}
		f_{i_0}^L(\hat x)-\frac{f_{i_0}^L(\bar x)}{g_{i_0}^U(\bar x)}g_{i_0}^U(\hat x)\geq 0 \ \ \text{or}\ \ f_{i_0}^U(\hat x)-\frac{f_{i_0}^U(\bar x)}{g_{i_0}^L(\bar x)}g_{i_0}^L(\hat x)\geq 0, 
	\end{equation*}
	or, equivalently,
	\begin{equation*}
		\frac{f_{i_0}^L(\hat x)}{g_{i_0}^U(\hat x)}\geq \frac{f_{i_0}^L(\bar x)}{g_{i_0}^U(\bar x)}  \ \ \text{or}\ \ \frac{f_{i_0}^U(\hat x)}{g_{i_0}^L(\hat x)}\geq\frac{f_{i_0}^U(\bar x)}{g_{i_0}^L(\bar x)}, 
	\end{equation*}
	This together with \eqref{equ:4} gives a contradiction, which completes the proof of (i).
	
	We now prove (ii). Suppose on the contrary that $\bar x\notin \mathcal{S}_1\eqref{problem}$. Then there exists $\hat x\in\Omega$ such that 
	\begin{equation*}
		\begin{cases}
			F_i(\hat x)\leq_{LU} F_i(\bar x),\ \ \forall i\in I,
			\\
			F_k(\hat x)<_{LU} F_k(\bar x),\ \ \text{for at least one}\ \ k\in I.
		\end{cases}
	\end{equation*}
	This implies that  $\hat x\neq \bar x$ and
	\begin{equation} \label{equa-5new}
		\frac{f_i^L(\hat x)}{g_i^U(\hat x)}\leq \frac{f_i^L(\bar x)}{g_i^U(\bar x)}\ \ \text{and}\ \ \frac{f_i^U(\hat x)}{g_i^L(\hat x)}\leq \frac{f_i^U(\bar x)}{g_i^L(\bar x)}, \ \ \forall i\in I,  
	\end{equation} 
	with at least one of the inequalities is strict. Hence, by the strictly generalized convexity of $(F, h)$ at $\bar x$ and the assumption that $(\lambda^L,\lambda^U)\neq (0,0)$, for $\hat x$ above, there exists $\nu\in [N(\bar x; S)]^\circ$ such that
	\begin{align}
		0&\leq \langle -\omega^*,\nu\rangle=\sum_{i\in I}\frac{\lambda_i^L}{g_i^U(\bar x)} \left[\langle x_i^{*L},\nu\rangle-\frac{f_i^L(\bar x)}{g_i^U(\bar x)}\langle y_i^{*U},\nu\rangle \right]+\sum_{i\in I}\frac{\lambda_i^U}{g_i^L(\bar x)} \left[\langle x_i^{*U},\nu\rangle-\frac{f_i^U(\bar x)}{g_i^L(\bar x)}\langle y_i^{*L},\nu\rangle\right]\notag
		\\
		&\ \ \ \ \ \ \ \ \ \ \ \ \ \ \ \ \ \ \ \ \ \ \ \ \ \ \ \ \ \ \ \ \ \ \ \ \ \ \ \ \ \ \ \ \ \ \ \ \ \ \ \ \ \ \ \ \ \ \ \ \ \ \ \ \ \ \ \ \ \ \ \ \ \ \ \ \ \ \ \ \ \ +\sum_{j\in J}\mu_j \langle z_j^*,\nu\rangle\notag
		\\
		&\ \ \ \ \ \ \ \ \ \ \ \ \ \ \ \ < \sum_{i\in I}\frac{\lambda_i^L}{g_i^U(\bar x)} \left[f_i^L(\hat x)-\frac{f_i^L(\bar x)}{g_i^U(\bar x)}g_i^U(\hat x) \right]+\sum_{i\in I}\frac{\lambda_i^U}{g_i^L(\bar x)} \left[f_i^U(\hat x)-\frac{f_i^U(\bar x)}{g_i^L(\bar x)}g_i^L(\hat x)\right]\notag
		\\
		&\ \ \ \ \ \ \ \ \ \ \ \ \ \ \ \ \ \ \ \ \ \ \ \ \ \ \ \ \ \ \ \ \ \ \ \ \ \ \ \ \ \ \ \ \ \ \ \ \ \ \ \ \ \ \ \ \ \ \ \ \ \ \ \ \ \ \ \ \ \ \ \ \ \ \ \ \ \ \ +\sum_{j\in J}\mu_j (h_j(\hat x)-h_j(\bar x))\notag
		\\
		&\ \ \ \ \ \ \ \ \ \ \ \ \ \ \ \ \leq \sum_{i\in I}\frac{\lambda_i^L}{g_i^U(\bar x)} \left[f_i^L(\hat x)-\frac{f_i^L(\bar x)}{g_i^U(\bar x)}g_i^U(\hat x) \right]+\sum_{i\in I}\frac{\lambda_i^U}{g_i^L(\bar x)} \left[f_i^U(\hat x)-\frac{f_i^U(\bar x)}{g_i^L(\bar x)}g_i^L(\hat x)\right].\notag
	\end{align} 
	This  implies that there exists $i_0\in I$ satisfying
	\begin{equation*}
		f_{i_0}^L(\hat x)-\frac{f_{i_0}^L(\bar x)}{g_{i_0}^U(\bar x)}g_{i_0}^U(\hat x)> 0 \ \ \text{or}\ \ f_{i_0}^U(\hat x)-\frac{f_{i_0}^U(\bar x)}{g_{i_0}^L(\bar x)}g_{i_0}^L(\hat x) > 0, 
	\end{equation*}
	or, equivalently,
	\begin{equation*}
		\frac{f_{i_0}^L(\hat x)}{g_{i_0}^U(\hat x)}> \frac{f_{i_0}^L(\bar x)}{g_{i_0}^U(\bar x)}  \ \ \text{or}\ \ \frac{f_{i_0}^U(\hat x)}{g_{i_0}^L(\hat x)}>\frac{f_{i_0}^U(\bar x)}{g_{i_0}^L(\bar x)}. 
	\end{equation*}
	It together with \eqref{equa-5new} gives a contradiction.    The proof is complete. 
\end{proof} 

\begin{remark}{\rm 
The condition \eqref{Necessary-Theorem} alone is not sufficient for Pareto solutions of \eqref{problem} if the (strict) generalized convexity of 	$(F, h)$  at  the point under consideration is violated. To see this, let us consider the following example.
	}
\end{remark}

\begin{example} {\rm 
We consider problem \eqref{problem} with $m=2,$ $n=p=1$, $S=(-\infty, 1]$, $h(x)=-x^2$, $f^L_1(x)=f_1^U(x)=-x^3+1$, $f^L_2(x)=f_2^U(x)=-2x^3+3$, $g_1^L(x)=g_2^L(x)=x^2+1$, and $g_1^U(x)=g_2^U(x)=x^2+2$. Then
\begin{equation*}
	F_1(x)=\left[\frac{-x^3+1}{x^2+2}, \frac{-x^3+1}{x^2+1}\right],  F_2(x)=\left[\frac{-2x^3+3}{x^2+2}, \frac{-2x^3+3}{x^2+1}\right],
\end{equation*} 
and $\Omega=S$. Let $\bar x=0\in \Omega$. Then, we have $N(\bar x; S)=\{0\}$ and 
$$\partial f_i^L(\bar x)=\partial f_i^U(\bar x)=\partial^+ g_i^L(\bar x)=\partial^+ g_i^U(\bar x)=\partial h(\bar x)=\{0\}, i\in \{1, 2\}.$$
Thus, $\bar x$ satisfies the KKT condition. However, since $\frac{1}{2}\in S$ and 
$$F\left(\frac{1}{2}\right)=\left(\left[\frac{7}{18},\frac{7}{10}\right], \left[\frac{11}{9},\frac{11}{5}\right]\right)<^s_{LU} F(\bar x)=\left(\left[\frac{1}{2},1\right], \left[\frac{3}{2},3\right]\right),$$ 
we arrive at $\bar x\notin \mathcal{S}_2^w\eqref{problem}$. 	Meanwhile, it is easy to check that $(F, h)$ is not generalized convex at $\bar x$.} 
\end{example} 
\section{Approximate duality theorems}\label{Duality-Relations}
Let $\mathcal{A}:=(A_1, \ldots, A_m)$ and $\mathcal{B}:=(B_1, \ldots, B_m)$, where $A_i$, $B_i$, $i\in I$, are intervals in $\mathcal{K}_c$. In what follows,   we use the following notations for convenience.
\begin{align*}
	\mathcal{A}&\preceq_{LU}\mathcal{B} \Leftrightarrow
	\begin{cases}
		A_i\leq_{LU} B_i, \ \ \forall i\in I,
		\\
		A_k<_{LU} B_k,  \ \ \text{for at least one}\ \ k\in I. 
	\end{cases}
	\\
	\mathcal{A}&\npreceq_{LU}\mathcal{B} \ \ \text{is the negation of}\ \ \mathcal{A}\preceq_{LU} \mathcal{B}.
	\\
	\mathcal{A}&\prec^s_{LU}\mathcal{B}  \Leftrightarrow A_i<^s_{LU} B_i, \ \ \forall i\in I.
	\\
	\mathcal{A}&\nprec^s_{LU}\mathcal{B} \ \ \text{is the negation of}\ \ \mathcal{A}\prec^s_{LU}\mathcal{B}. 
\end{align*}

For $y\in\mathbb{R}^n$, $(\lambda^L, \lambda^U)\in(\mathbb{R}^m_+\times\mathbb{R}^m)_+\setminus\{(0,0)\}$, and $\mu\in \mathbb{R}^{p}_+$, put
\begin{align*}
	\mathcal{L}(y, \lambda^L,\lambda^U, \mu):=F(y)=\left(F_1(y), \ldots, F_m(y)\right),
\end{align*} 
where \begin{equation*}
	F_i(y):=\frac{f_i(y)}{g_i(y)}=\left[\frac{f_i^L(y)}{g_i^U(y)}, \frac{f_i^U(y)}{g_i^L(y)}\right],\ \ i\in I.
\end{equation*}

In connection with the primal problem \eqref{problem}, we consider the following dual problem in the sense of Mond--Weir:
\begin{align}
	\label{Dual-problem}\tag{FIMD$_{MW}$} 
	LU-&\max\ \  \mathcal{L}(y, \lambda^L,\lambda^U, \mu) 
	\\
	&\text{s.t.}\ \ (y, \lambda^L,\lambda^U, \mu)\in\Omega_{MW},\notag
\end{align}
where the feasible set $\Omega_{MW}$ is defined by
\begin{align*}
	\Omega&_{MW}:=\Big\{(y, \lambda^L, \lambda^U, \mu)\in S\times\mathbb{R}^m_+ \times\mathbb{R}^m_+\times\mathbb{R}^{p}_+\,:\, 0\in \sum_{i\in I}\frac{\lambda_i^L}{g_i^U(y)} \left[\partial f_i^L(y)-\frac{f_i^L(y)}{g_i^U(y)}\partial^+ g_i^U(y)\right]
	\\
	&\ \ \ \ \ \ \ \ \ \ \ \ \ \ \ \ \ \ \ \ \ \ \ \ \ \ \ \ +\sum_{i\in I}\frac{\lambda_i^U}{g_i^L(y)} \left[\partial f_i^U(y)-\frac{f_i^U(y)}{g_i^L(y)}\partial^+ g_i^L(y)\right]+\sum_{j\in J}\mu_j\partial h_j(y)+N(y; S),
	\\
	&\ \ \ \ \ \ \ \ \ \ \ \ \ \ \ \ \ \ \ \ \ \ \ \ \ \ \ \ \ \sum_{j\in J}\mu_j h_j(y)\geq 0, \ \  \sum_{i\in I}(\lambda_i^L+\lambda_i^U)+\sum_{j\in J}\mu_j=1, (\lambda^L,\lambda^U)\neq (0,0)\Big\}.
\end{align*} 

\begin{definition}
	{\rm Let $(\bar y, \bar \lambda^L, \bar \lambda^U, \bar \mu)\in \Omega_{MW}$. We say that
		\begin{enumerate}[(i)]
			\item  	$(\bar y, \bar \lambda^L, \bar \lambda^U, \bar \mu)$ is a {\em type-1 Pareto solution} of \eqref{Dual-problem}, denoted by 
			$$(\bar y, \bar \lambda^L, \bar \lambda^U, \bar \mu)\in \mathcal{S}_1\eqref{Dual-problem},$$ 
			if there is no $(y, \lambda^L, \lambda^U, \mu)\in \Omega_{MW}$ such that $$\mathcal{L}(\bar y, \bar \lambda^L, \bar \lambda^U, \bar \mu)\preceq_{LU}\mathcal{L}(y, \lambda^L, \lambda^U, \mu).$$
			\item  	$(\bar y, \bar \lambda^L, \bar \lambda^U, \bar \mu)$ is a {\em type-2 weakly Pareto solution} of \eqref{Dual-problem}, denoted by 
			$$(\bar y, \bar \lambda^L, \bar \lambda^U, \bar \mu)\in \mathcal{S}_2^w\eqref{Dual-problem},$$ 
			if there is no $(y, \lambda^L, \lambda^U, \mu)\in \Omega_{MW}$ such that $$\mathcal{L}(\bar y, \bar \lambda^L, \bar \lambda^U, \bar \mu)\preceq^s_{LU}\mathcal{L}(y, \lambda^L, \lambda^U, \mu).$$
		\end{enumerate}	
		
	}
\end{definition}

The following theorem describes weak duality relations  between the primal problem \eqref{problem} and the dual problem \eqref{Dual-problem}.
\begin{theorem}[Weak duality]\label{weak-dual} Let $x\in \Omega$ and $(y, \lambda^L, \lambda^U, \mu)\in\Omega_{MW}$.
	\begin{enumerate}[\rm(i)]
		\item If $(F, h)$ is generalized convex on $S$ at $y$, then
		\begin{equation*}
			F(x) \nprec^s_{LU} \mathcal{L}(y, \lambda^L, \lambda^U, \mu).
		\end{equation*}
		\item If $(F, h)$ is strictly generalized convex on $S$ at $y$, then
		\begin{equation*}
			F(x)\npreceq_{LU} \mathcal{L}(y, \lambda^L, \lambda^U, \mu).
		\end{equation*}
	\end{enumerate}
\end{theorem}
\begin{proof} Since $x\in\Omega$ and $(y, \lambda^L, \lambda^U, \mu)\in\Omega_{MW}$,  we have $x, y\in S$,  
	\begin{equation}\label{equ:7}
		h_j(x)\leq 0, \ \ \sum_{j\in J}\mu_jh_j(y)\geq 0,  
	\end{equation}
	and 
	\begin{align*}
		0\in \sum_{i\in I}\frac{\lambda_i^L}{g_i^U(y)} \left[\partial f_i^L(y)-\frac{f_i^L(y)}{g_i^U(y)}\partial^+ g_i^U(y)\right]+\sum_{i\in I}\frac{\lambda_i^U}{g_i^L(y)} &\left[\partial f_i^U(y)-\frac{f_i^U(y)}{g_i^L(y)}\partial^+ g_i^L(y)\right]
		\\
		& +\sum_{j\in J}\mu_j\partial h_j(y)+N(y; S).
	\end{align*}
	This implies that there exist $x_i^{*L}\in \partial f_i^L(y)$, $x_i^{*U}\in \partial f_i^U(y)$, $y_i^{*L}\in \partial^+ g_i^L(y)$, $y_i^{*U}\in \partial^+ g_i^U(y)$, $i\in I$, $z^*_j\in\partial h_j(y)$, $j\in J$, and $\omega^*\in N(y; S)$ such that
	\begin{equation}\label{equ:6}
		\sum_{i\in I}\frac{\lambda_i^L}{g_i^U(y)} \left[x_i^{*L}-\frac{f_i^L(y)}{g_i^U(y)}y_i^{*U}\right]+\sum_{i\in I}\frac{\lambda_i^U}{g_i^L(y)} \left[x_i^{*U}-\frac{f_i^U(y)}{g_i^L(y)}y_i^{*L}\right]+\sum_{j\in J}\mu_j z_j^*=-\omega^*.
	\end{equation}
	
	We first prove (i). Suppose on the contrary that
	\begin{equation*}
		F(x)\prec^s_{LU} \mathcal{L}(y, \lambda^L, \lambda^U, \mu),
	\end{equation*}
	or, equivalently,
	\begin{equation*}
		F_i(x)\prec^s_{LU} \mathcal{L}_i(y, \lambda^L, \lambda^U, \mu), \ \ \forall i\in I.
	\end{equation*}
	Then, 
	\begin{equation}\label{equ:10}
		\frac{f_i^L(x)}{g_i^U(x)}<\frac{f_i^L(y)}{g_i^U(y)}\ \ \text{and}\ \ \frac{f_i^U(x)}{g_i^L(x)}<\frac{f_i^U(y)}{g_i^L(y)}, \ \ \forall i\in I.
	\end{equation}
	By the generalized convex property of $(F, h)$ on $S$ at $y$, for such $x$, there exists $\nu\in [N(y, S)]^\circ$ such that
	\begin{equation*} 
		\begin{aligned}
			&\sum_{i\in I}\frac{\lambda_i^L}{g_i^U(y)} \left[\langle x_i^{*L},\nu\rangle-\frac{f_i^L(y)}{g_i^U(y)}\langle y_i^{*U},\nu\rangle \right]+\sum_{i\in I}\frac{\lambda_i^U}{g_i^L(y)} \left[\langle x_i^{*U},\nu\rangle-\frac{f_i^U(y)}{g_i^L(y)}\langle y_i^{*L},\nu\rangle\right]+\sum_{j\in J}\mu_j \langle z_j^*,\nu\rangle
			\\
			&\leq \sum_{i\in I}\frac{\lambda_i^L}{g_i^U(y)} \left[f_i^L(x)-f_i^L(y)-\frac{f_i^L(y)}{g_i^U(y)}(g_i^U(x)-g_i^U(y)) \right]
			\\
			&\ \ \ \ +\sum_{i\in I}\frac{\lambda_i^U}{g_i^L(y)} \left[f_i^U(x)-f_i^U(y)-\frac{f_i^U(y)}{g_i^L(y)}(g_i^L(x)-g_i^L(y))\right]+\sum_{j\in J}\mu_j (h_j(x)-h_j(y))
			\\
			&=\sum_{i\in I}\frac{\lambda_i^L}{g_i^U(y)} \left[f_i^L(x)-\frac{f_i^L(y)}{g_i^U(y)}g_i^U(x) \right]+\sum_{i\in I}\frac{\lambda_i^U}{g_i^L(y)} \left[f_i^U(x)-\frac{f_i^U(y)}{g_i^L(y)}g_i^L(x)\right]
			\\
			&\ \ \ \ \ \ \ \ \ \ \ \ \ \ \ \ \ \ \ \ \ \ \ \ \ \ \ \ \ \ \ \ \ \ \ \ \ \ \ \ \ \ \ \ \ \ \; +\sum_{j\in J}\mu_j (h_j(x)-h_j(y)).
		\end{aligned}
	\end{equation*} 
	It follows from \eqref{equ:6} and relations $\omega^*\in N(y; S)$ and $\nu\in [N(y, S)]^\circ$ that  
	\begin{equation*} 
		\begin{aligned}
			0\leq \sum_{i\in I}\frac{\lambda_i^L}{g_i^U(y)} \Big[\langle x_i^{*L},\nu\rangle-\frac{f_i^L(y)}{g_i^U(y)}\langle y_i^{*U},\nu\rangle \Big]+\sum_{i\in I}\frac{\lambda_i^U}{g_i^L(y)} \Big[\langle x_i^{*U},\nu\rangle&-\frac{f_i^U(y)}{g_i^L(y)}\langle y_i^{*L},\nu\rangle\Big]
			\\
			&+\sum_{j\in J}\mu_j \langle z_j^*,\nu\rangle.
		\end{aligned}
	\end{equation*}
	Thus,
	\begin{equation*} 
		\begin{aligned}
			0\leq\sum_{i\in I}\frac{\lambda_i^L}{g_i^U(y)} \left[f_i^L(x)-\frac{f_i^L(y)}{g_i^U(y)}g_i^U(x) \right]+\sum_{i\in I}\frac{\lambda_i^U}{g_i^L(y)} &\left[f_i^U(x)-\frac{f_i^U(y)}{g_i^L(y)}g_i^L(x)\right]
			\\
			&+\sum_{j\in J}\mu_j (h_j(x)-h_j(y)).
		\end{aligned}
	\end{equation*}
	It together with \eqref{equ:7} implies that
	\begin{equation}\label{equ:8}
		0\leq\sum_{i\in I}\frac{\lambda_i^L}{g_i^U(y)} \left[f_i^L(x)-\frac{f_i^L(y)}{g_i^U(y)}g_i^U(x) \right]+\sum_{i\in I}\frac{\lambda_i^U}{g_i^L(y)} \left[f_i^U(x)-\frac{f_i^U(y)}{g_i^L(y)}g_i^L(x)\right].
	\end{equation} 
	By \eqref{equ:8} and the fact that  $(\lambda^L,\lambda^U)\neq (0,0)$, it follows that there is $i_0\in I$ such that
	\begin{equation*}
		f_{i_0}^L(x)-\frac{f_{i_0}^L(y)}{g_{i_0}^U(y)}g_{i_0}^U(x)\geq 0 \ \  \text{or}  \ \ f_{i_0}^U(x)-\frac{f_{i_0}^U(y)}{g_{i_0}^L(y)}g_{i_0}^L(x)\geq 0,
	\end{equation*}
	or, equivalently,
	\begin{equation*}
		\frac{f_{i_0}^L(x)}{g_{i_0}^U(x)}\geq\frac{f_{i_0}^L(y)}{g_{i_0}^U(y)} \ \  \text{or}  \ \ \frac{f_{i_0}^U(x)}{g_{i_0}^L(x)}\geq\frac{f_{i_0}^U(y)}{g_{i_0}^L(y)},
	\end{equation*} 
	which contradicts \eqref{equ:10} and therefore completes the proof of (i).
	
	Next we prove (ii). Assume to the contrary that 
	\begin{equation*}
		F(x)\preceq_{LU} \mathcal{L}(y, \lambda^L, \lambda^U, \mu).
	\end{equation*} 
	This means that 
	\begin{equation}\label{equ:11}
		\begin{cases}
			F_i(x)\leq_{LU} F_i(y), \ \ \forall i\in I,
			\\
			F_k(x)<_{LU} F_k(y), \ \ \text{for at least one} \ \ k\in I.
		\end{cases}
	\end{equation}
	Hence, $x\neq y$. By the strictly of $(F, h)$ on $S$ at $y$ and the assumption that $(\lambda^L, \lambda^U)\neq (0,0)$, for such $x$, there exists $\nu\in [N(y, S)]^\circ$ such that
	\begin{equation*} 
		\begin{aligned}
			&\sum_{i\in I}\frac{\lambda_i^L}{g_i^U(y)} \left[\langle x_i^{*L},\nu\rangle-\frac{f_i^L(y)}{g_i^U(y)}\langle y_i^{*U},\nu\rangle \right]+\sum_{i\in I}\frac{\lambda_i^U}{g_i^L(y)} \left[\langle x_i^{*U},\nu\rangle-\frac{f_i^U(y)}{g_i^L(y)}\langle y_i^{*L},\nu\rangle\right]+\sum_{j\in J}\mu_j \langle z_j^*,\nu\rangle
			\\
			&< \sum_{i\in I}\frac{\lambda_i^L}{g_i^U(y)} \left[f_i^L(x)-f_i^L(y)-\frac{f_i^L(y)}{g_i^U(y)}(g_i^U(x)-g_i^U(y)) \right]
			\\
			&\ \ \ \ +\sum_{i\in I}\frac{\lambda_i^U}{g_i^L(y)} \left[f_i^U(x)-f_i^U(y)-\frac{f_i^U(y)}{g_i^L(y)}(g_i^L(x)-g_i^L(y))\right]+\sum_{j\in J}\mu_j (h_j(x)-h_j(y))
			\\
			&=\sum_{i\in I}\frac{\lambda_i^L}{g_i^U(y)} \left[f_i^L(x)-\frac{f_i^L(y)}{g_i^U(y)}g_i^U(x) \right]+\sum_{i\in I}\frac{\lambda_i^U}{g_i^L(y)} \left[f_i^U(x)-\frac{f_i^U(y)}{g_i^L(y)}g_i^L(x)\right]
			\\
			&\ \ \ \ \ \ \ \ \ \ \ \ \ \ \ \ \ \ \ \ \ \ \ \ \ \ \ \ \ \ \ \ \ \ \ \ \ \ \ \ \ \ \ \ \ \ \; +\sum_{j\in J}\mu_j (h_j(x)-h_j(y)).
		\end{aligned}
	\end{equation*} 
	Continuing a similar procedure as in the proof of (i), we arrive at
	\begin{equation*} 
		0<\sum_{i\in I}\frac{\lambda_i^L}{g_i^U(y)} \left[f_i^L(x)-\frac{f_i^L(y)}{g_i^U(y)}g_i^U(x) \right]+\sum_{i\in I}\frac{\lambda_i^U}{g_i^L(y)} \left[f_i^U(x)-\frac{f_i^U(y)}{g_i^L(y)}g_i^L(x)\right].
	\end{equation*} 
	Hence, there exists $i_0\in I$ such that 
	\begin{equation*}
		\frac{f_{i_0}^L(x)}{g_{i_0}^U(x)}>\frac{f_{i_0}^L(y)}{g_{i_0}^U(y)} \ \  \text{or}  \ \ \frac{f_{i_0}^U(x)}{g_{i_0}^L(x)}>\frac{f_{i_0}^U(y)}{g_{i_0}^L(y)}.
	\end{equation*}
	It together with \eqref{equ:11} gives a contradiction, which completes the proof.  
\end{proof}
The following example asserts the importance of the generalized convexity of $(F,h)$ on $S$ used in Theorem \ref{weak-dual}.This means that the conclusion of Theorem \ref{weak-dual} may fail if this property has been violated.  

\begin{example}\label{Example-3}{\rm We consider problem \eqref{problem} with $m=2,$ $n=p=1$, $S=(-\infty, 1]$, $h(x)=-|x|$,  $f^L_1(x)=f_1^U(x)=1-x^3$, $f^L_2(x)=f_2^U(x)=1-x^5$, $g_1^L(x)=g_2^L(x)=x^2+1$, and $g_1^U(x)=g_2^U(x)=x^2+2$. Then
		\begin{equation*}
			F_1(x)=\left[\frac{1-x^3}{x^2+2}, \frac{1-x^3}{x^2+1}\right],  F_2(x)=\left[\frac{1-x^5}{x^2+2}, \frac{1-x^5}{x^2+1}\right],
		\end{equation*} 
and $\Omega=S$. Let $\bar x=1\in \Omega$. We now consider the dual problem \eqref{Dual-problem}. By choosing $\bar y=0\in S$, $\bar\lambda_1^L=\bar\lambda_2^L=\bar\lambda_1^U=\bar\lambda_2^U=\frac{1}{4}$, $\bar\mu=0$, we have $(\bar y, \bar\lambda^L,\bar\lambda^U,\bar\mu)\in\Omega_{MW}$ and that
\begin{equation*}
F(\bar x)=\big([0,0], [0,0]\big)\prec^s_{LU}\mathcal{L}(\bar y, \bar\lambda^L,\bar\lambda^U,\bar\mu)=\left(\left[\frac{1}{2}, 1\right], \left[\frac{1}{2}, 1\right]\right).
\end{equation*}  	
The reason is that the generalized convexity of $(F,h)$  on $S$ has been violated  at $\bar y$.}
\end{example}  

Next we present a theorem that formulates strong duality relations between the primal problem \eqref{problem} and the dual problem \eqref{Dual-problem}.
\begin{theorem}[Strong duality]\label{Strong-duality} Suppose that $\bar x\in \mathcal{S}_2^w\eqref{problem}$ and the \eqref{CQ} is satisfied at this point. Then there exist $(\bar \lambda^L, \bar\lambda^U)\in (\mathbb{R}^m_+\times\mathbb{R}^m_+)\setminus\{(0,0)\}$, and $\bar \mu\in \mathbb{R}^p_+$ such that $(\bar x, \bar \lambda^L, \bar \lambda^U, \bar \mu)\in \Omega_{MW}$ and $F(\bar x)=\mathcal{L}(\bar x, \bar \lambda^L, \bar\lambda^U, \bar \mu)$. Furthermore,
	\begin{enumerate}[\rm(i)]
		\item If $(F, h)$ is generalized convex on $S$ at $\bar x$, then $(\bar x, \bar \lambda^L, \bar \lambda^U, \bar \mu)$ is a type-2 weakly Pareto solution of \eqref{Dual-problem}.
		\item If $(F, h)$  is strictly  generalized convex on $S$ at $\bar x$, then $(\bar x, \bar \lambda^L, \bar \lambda^U, \bar \mu)$ is a type-1 Pareto solution of \eqref{Dual-problem}.
	\end{enumerate}
\end{theorem}
\begin{proof} Since $\bar x\in \mathcal{S}_2^w\eqref{problem}$ and the \eqref{CQ} is satisfied at this point, by Theorem \ref{KKT-Theorem}, there exist $(\bar \lambda^L, \bar\lambda^U)\in (\mathbb{R}^m_+\times\mathbb{R}^m_+)\setminus\{(0,0)\}$  and $\bar \mu\in \mathbb{R}^p_+$ such that  $(\bar x, \bar \lambda^L, \bar \lambda^U, \bar \mu)\in \Omega_{MW}$. Clearly, 
	$$F(\bar x)=\mathcal{L}(\bar x, \bar\lambda^L, \bar\lambda^U, \bar\mu).$$
	
	(i) Since  $(F, h)$ is  generalized convex on $S$ at $\bar x$, by Theorem \ref{weak-dual}(i), we have
	\begin{equation*}
		F(\bar x)=\mathcal{L}(\bar x, \bar\lambda^L, \bar\lambda^U, \bar\mu)\nprec^s_{LU} L(y,\lambda^L,\lambda^U, \mu) 
	\end{equation*}    
	for all $(y,\lambda^L,\lambda^U, \mu)\in\Omega_{MW}$. This means that $(\bar x, \bar\lambda^L, \bar\lambda^U, \bar\mu)$ is a type-2 weakly Pareto solution of \eqref{Dual-problem}. 
	
	(ii)   If $(F, h)$ is  strictly generalized convex on $S$ at $\bar x$,  then by invoking Theorem \ref{weak-dual}(ii), we obtain
	\begin{equation*}
		F(\bar x)=\mathcal{L}(\bar x, \bar\lambda^L, \bar\lambda^U, \bar\mu)\npreceq_{LU} L(y,\lambda^L,\lambda^U, \mu) 
	\end{equation*}    
	for all $(y,\lambda^L,\lambda^U, \mu)\in\Omega_{MW}$. Thus, $(\bar x, \bar\lambda^L, \bar\lambda^U, \bar\mu)$ is a type-1  Pareto solution of \eqref{Dual-problem}, which completes the proof.
\end{proof}
\begin{remark}
{\rm 
The \eqref{CQ} condition plays an important role in establishing the strong duality results in Theorem \ref{Strong-duality}. This means that if the \eqref{CQ} is not satisfied at a type-2 weakly Pareto solution of \eqref{problem}, then strong dual relations in Theorem \ref{Strong-duality} are no longer true at this point.   Indeed, let us look back at Example \ref{Not-CQ}. We see that $\bar x=0\in \mathcal{S}^w_2\eqref{problem}$. Furthermore,  the \eqref{CQ} is not satisfied at $\bar x$ and there do not exist a triple $(\bar\lambda^L, \bar\lambda^U, \bar\mu)$ such that $(\bar x, \bar\lambda^L, \bar\lambda^U, \bar\mu)\in \Omega_{MW}$. Thus, in this case, we do not have strong dual relations. 

}
\end{remark}
We finish this section by establishing converse-like duality relations for  Pareto solutions between the primal problem \eqref{problem}  and the dual one  \eqref{Dual-problem}.
\begin{theorem}[Converse-like duality] Let $(\bar x, \bar\lambda^L, \bar\lambda^U, \bar\mu)\in \Omega_{MW}$.
	\begin{enumerate}[\rm(i)]
		\item If $\bar x\in\Omega$ and $(F, h)$ is  generalized convex on $S$ at $\bar x$, then $\bar x$ is a type-2  weakly Pareto solution of \eqref{problem}. 
		\item If $\bar x\in\Omega$ and $(F, h)$ is strictly  generalized convex on $S$ at $\bar x$, then $\bar x$ is a type-1  Pareto solution of \eqref{problem}.
	\end{enumerate}	
\end{theorem}
\begin{proof} (i) Since $(\bar x, \bar\lambda^L, \bar\lambda^U, \bar\mu)\in \Omega_{MW}$, we have
	\begin{align}
		&0\in \sum_{i\in I}\frac{\lambda_i^L}{g_i^U(\bar x)} \left[\partial f_i^L(\bar x)-\frac{f_i^L(\bar x)}{g_i^U(\bar x)}\partial^+ g_i^U(\bar x)\right]+\sum_{i\in I}\frac{\lambda_i^U}{g_i^L(\bar x)} \left[\partial f_i^U(\bar x)-\frac{f_i^U(\bar x)}{g_i^L(\bar x)}\partial^+ g_i^L(\bar x)\right]\notag
		\\
		&\ \ \ \ \ \ \ \ \ \ \ \ \ \ \ \ \ \ \ \ \ \ \ \ \ \ \ \ \ \ \ \ \ \ \ \ \ \ \ \ \ \ \ \ \ \ \ \ \ \ \ \  +\sum_{j\in J}\mu_j\partial h_j(\bar x)+N(\bar x; S),\notag
		\\
		&\sum_{j\in J}\mu_j h_j(\bar x)\geq 0, \ \  \sum_{i\in I}(\lambda_i^L+\lambda_i^U)+\sum_{j\in J}\mu_j=1, (\lambda^L,\lambda^U)\neq (0,0).\label{equ:12}
	\end{align} 
	
	Since $\bar x\in\Omega$, one has $h_j(\bar x)\leq 0$ for all $j\in J$. Hence, $\mu_jh_j(\bar x)\leq 0$ for all $j\in J$. This together with \eqref{equ:12} yields $\mu_jh_j(\bar x)= 0$ for all $j\in J$. Thus, by the  generalized convexity of $(F, h)$ on $S$ at $\bar x$ and  Theorem \ref{Sufficient-Theorem}(i), $\bar x$ is a type-2  weakly Pareto solution of \eqref{problem}. 
	
	The proof of (ii) is similar to that of (i) by using the strictly  generalized convexity of $(F, h)$  and Theorem \ref{Sufficient-Theorem}(ii), so omitted.
\end{proof}

\vskip 6mm
\noindent{\bf Acknowledgments}

\noindent  The authors would like to thank the two anonymous referees, whose suggestions and comments improved the paper. This research is funded by Hanoi Pedagogical University 2 under grant number HPU2.UT-2021.15.

\end{document}